\documentclass[a4paper,11pt]{amsart}
\usepackage[english]{babel}
\usepackage{amsmath,amssymb,amsthm,mathrsfs,amsopn,amscd}

\usepackage{verbatim,graphicx,color}

\newtheorem{thm}{Theorem}[section]
\newtheorem{lem}[thm]{Lemma}
\newtheorem{pro}[thm]{Proposition}
\newtheorem{cor}[thm]{Corollary}

\newtheorem{remark}[thm]{Remark}

\newcommand{\abs}[1]{\lvert#1\rvert}

\DeclareMathOperator{\sing}{sing}

\DeclareMathOperator{\Div}{div}
\DeclareMathOperator{\dist}{dist}
\DeclareMathOperator{\spt}{spt}

\newcommand{\Sp}{\mathbb{S}}
\newcommand{\N}{\mathbb{N}}

\newcommand{\E}{\mathcal{E}}

\def\R{\mathbb R}

\def\pa{\partial}

\newenvironment{ack}{{\bf Acknowledgements. }}

\title[]{Lower bound for the perimeter density at singular points of a minimizing cluster in $\R^N$}

\author{Jonas Hirsch}
\address{Scuola Internazionale Superiore di Studi Avanzati, via Bonomea 265, 34136 Trieste, ITALY}
\email{jhirsch@sissa.it}

\author{Michele Marini}
\address{Scuola Internazionale Superiore di Studi Avanzati, via Bonomea 265, 34136 Trieste, ITALY}
\email{mmarini@sissa.it}


\begin{document}

\maketitle

\begin{abstract}
In this paper we study the blow-ups of the singular points in the boundary of a minimizing cluster lying in the interface of more than two chambers. We establish a sharp lower bound for the perimeter density at those points and we prove that this bound is rigid, namely having the lowest possible density completely characterizes the blow-up.
\end{abstract}

\section{Introduction}

This paper deals with the study of the singularities of an {\it isoperimetric cluster}, that is a finite and disjoint family of sets of finite perimeter, called chambers, which minimizes the sum of the perimeter of all the chambers (Section \ref{sec.prel} contains all the precise definitions).\\

If the cluster has only one chamber, it is well known that the ball is the isoperimetric set (with volume constraint).
As the number of chambers increases, the problem becomes slightly more difficult: while existence and regularity of isoperimetric clusters is nowadays classical--see \cite{Al}, we also refer to \cite{Ma} and the references therein for a complete overview on the subject--very few is still known about the shape of those minimizing clusters.
A complete characterization of them is obtained only in the case when there are two chambers (see \cite{FABHZ,HMRR,HLR, Re}), or, in the plane, when there are three chambers (see \cite{Wi}).
\par
Since it seems to be very difficult to obtain such a "global" description of the isoperimetric clusters with an arbitrary number of chambers, a crucial role is played by a "local" analysis of their boundary. \\

In \cite{Al} it is proved, in analogy to the case when there is only one chamber, that the reduced boundary of the cluster (we refer again to Section \ref{sec.prel} for the definition) is a finite union of  analytic hypersurfaces with constant mean curvature, while, as we shall explain, the problem of the description of the singular part of the boundary of the cluster is very interesting and almost completely open.
\par
Just to make an example of the peculiar behavior of the singular points when dealing with clusters, let us notice that, while minimal surfaces can develop singularities only in dimension $N\ge 8$, in the case of minimizing clusters, any point in the interface of three or more different chambers is a singular point in every dimension.\\ 
As we shall see in the following sections these kind of singularities produced by the junction of three or more chambers will play a crucial role in our analysis. We define
\[
c\mathrm{-sing}(\mathcal E)=\left\{\text{points lying in the boundary of at least three chambers of }\mathcal E\right\}.
\]
We will concentrate our analysis to theese points.\\

Given a minimizing cluster $\mathcal E$ and the singular point $p\in\partial\mathcal E$, the monotonicity formula \eqref{monoton} introduced in Section \ref{sec.prel}, ensures that the blow-up $\lim_{r\to 0}r^{-1}\left(\mathcal E-p\right)$ is still minimizing and that it is a cone. The study of those blow-ups, and thence of cone-like minimising clusters, is then a natural tool in order to get a local description of the boundary of the cluster around a singular point.
In dimension $N=2$, there is only one (up to rotations) possible blow-up of the boundary of the cluster around a singular point. This blow-up consists of three half-line meeting at the origin and forming three 120-degrees angles. Throughout this paper we will denote this set simply by $Y$.
In $\R^3$, a remarkable result in \cite{Ta} provides us with a complete characterization of all the possible blow-ups around singular points asserting that there are only two possible cases: the set $Y\times\R$ and the cone generated by the centre of a regular tetrahedron and his edges.
\par
The singular set of minimal surfaces $L^\infty$-close to $Y\times \R^{N-2}$, or to a tetrahedron with an $N-3$ spine had been investigated in the seminal works  \cite{SimonCone} and by \cite{ColSpol}.
\par
To our knowledge, no characterization of cones is settled in dimension greater than three.\\

The description of the possible blow-ups deeply relies on the ambient dimension $N$. Nevertheless we shall prove a dimension-free lower bound for the perimeter density $\Theta_p(\mathcal E)$ defined as the perimeter of $\mathcal E$ inside the ball of radius $r$ divided by the volume of the $N-1$-dimensional ball of the same radius.
More precisely in Section \ref{sec.ldb} we prove the following theorem:

\begin{thm}\label{thm.lowdens}
Let $N\ge 2$, and let $\mathcal E$ be a cone-like minimizing cluster with at least three chambers. Then 
\[
\Theta_0(\mathcal E)\ge \frac 3 2.
\]
\end{thm}

Theorem \ref{thm.lowdens} can be rephrased as follows: $Y\times \R^{N-2}$ has the lowest possible perimeter density at the origin, among all minimizing cone-like clusters of $\R^N$, with at least three chambers.
As we shall see in Section \ref{sec.consequences} the above statement is rigid, in particular, in Corollary \ref{cor.characterizationbydensity} we show that any other cone-like minimizing cluster has density strictly greater than $3/2$.\\

To prove Theorem \ref{thm.lowdens} we use an induction argument on the dimension in order to reduce to the case $N=2$, where the conclusion easily follows by the above mentioned characterization of the minimizing cone-like clusters. To prove the inductive step we show the following:
\begin{itemize}
\item[i)] there always exist singular points in the unit sphere where more than two chamber meets,\\
 \item[ii)]the blow-up at those points is a cone of the form $\mathcal E'\times \R$, with $\mathcal E'\subset\R^{N-1}$ with density lower than $\Theta_0(\mathcal E)$.
\end{itemize}

The proof of item i) requires some preliminary results introduced in Section \ref{sec.connected}.\\

In Section \ref{sec.consequences}, we consider a converging sequence of minimizing clusters $\mathcal E_k\rightarrow \mathcal E$ and we show, as an application of Theorem \ref{thm.lowdens}, that $c\mathrm{-sing}(\mathcal E_k)$ converges in Hausdorff to $c\mathrm{-sing}(\mathcal E)$ (see Proposition \ref{prop.sequential}).
It is relatively easy to show that a sequence of $c$-singular points must converge to a $c$-singular point of the limit cluster, but it is more difficult to show that every $c$-singularity cannot appear just in the limit.
To show this, we reduce to the case $N=2$ and we prove in Lemma \ref{lem.Ynonremovable}, that every cluster $\mathcal E\subset \R^2$ sufficiently close to $Y$ must have a $c$-singular point in a neighborhood of the origin.

\section{Preliminaries and basic definitions}\label{sec.prel}
Let $m\in\mathbb N$ and let $\mathcal E=\left\{\mathcal E(i)\right\}_{i=1}^m$ be a family of subset of $\R^N$. We say that $\mathcal E$ is a $m$-{\it cluster} (or simply a {\it cluster}), provided each $\mathcal E(i)$ is a set of locally finite perimeter and $|\mathcal E(i)\cap\mathcal E(j)|=0$, for every $i\neq j$, where $|\cdot|$ denotes the $N$-dimensional Lebesgue measure. We call the sets $\mathcal E(i)$ the {\it chambers} of the cluster $\mathcal E$, and we denote by $\mathcal E(0)$ the {\it exterior chamber} $\R^N\setminus \overset{m}{\underset{i=1}{\bigcup}}\mathcal E(i)$ .\\
Given a $m$-cluster $\mathcal E$ and $i\neq j\in\{0,\ldots,m\}$ we denote by $\mathcal E(i,j)$ the {\it interface}  of the chambers $\mathcal E(i)$ and $\mathcal E(j)$ defined as the intersection of their reduced boundaries, namely $\mathcal E(i,j)=\partial^*\mathcal E(i)\cap\partial^*\mathcal E(j)$. 
Given an open set $A$, we finally define the {\it relative perimeter} of $\mathcal E$ in $A$ as
\[
P(\mathcal E;A)={\underset{0\le i<j\le m}{\sum}}\mathcal H^{N-1}\left(\mathcal E(i,j)\cap A\right),
\]
where $\mathcal H^{N-1}$ denotes the $N-1$-dimensional Hausdorff measure. By virtue of \cite[Proposition 29.4]{Ma} the above definition is equivalent to
\begin{equation}\label{eq:1/2perimeter}
P(\mathcal E;A)=\frac 1 2 \sum_{i=0}^m P(\mathcal E(i);A).
\end{equation}
Here we denoted by $P(E;A)$ the relative perimeter of a set $E$ of locally finite perimeter in $A$.\\

Let $\mathcal E$ and $\mathcal F$ be two $m$-clusters and let $\mathrm{d}(\mathcal E,\mathcal F)=\sum_{i=1}^m|\mathcal E(i)\Delta\mathcal F(i)|$ be the sum of the volumes of the symmetric differences between the chambers of $\mathcal E$ and the chambers of $\mathcal F$.
Given $\Lambda,\,\rho>0$, we say that $\mathcal E$ is a $(\Lambda, \rho)$-{\it minimizing cluster} in $A$ if 
\[
P(\mathcal E;B_r(x))\le P(\mathcal F; B_r(x))+\Lambda \mathrm{d}(\mathcal E,\mathcal F),
\]
holds true for every $r<\rho$ and every $x\in\R^N$ such that $\mathcal E(i)\Delta \mathcal F(i)\subset\subset B_r(x)$.\\
We stress that, if a cluster $\mathcal E$ is a minimizer of the perimeter functional among all clusters of given volume, namely
\[
P(\mathcal E)=\inf \left\{P(\mathcal F)\,:\,|\mathcal F(i)|=|\mathcal E(i)|\right\},
\]
then $\mathcal E$ is a $(\Lambda,\rho)$-minimizing cluster, for some positive numbers $\Lambda$ and $\rho$ (see for instance \cite[Chapter 29]{Ma}).\\

If $\mathcal E$ is a $(\Lambda,\rho)$-minimizing cluster, then its reduced boundary is a $C^{1,\alpha}$-hypersurface, for every $\alpha\in(0,1)$ (see \cite[Corollary 4.6]{CLM}); however, the set
\[
\partial \mathcal E:=\mathrm{cl}\left(\overset{N}{\underset{i=1}{\bigcup}}\partial^*\mathcal E(i)\right)=\overset{N}{\underset{i=1}{\bigcup}}\left\{x\,:\,0<|B_r(x)\cap\mathcal E(i)|<r^N\omega_N, \ \  r>0\right\}
\]
can develop singularities.\\
We define the $c$-{\it singular set} as the subset of $\partial E\setminus \overset{N}{\underset{i=1}{\bigcup}}\partial^*\mathcal E(i)$ consisting of those singularities arising in the junction of three or more chambers, more precisely we set
\[
c\mathrm{-sing}(\mathcal E)=\{x\in\partial\mathcal E\,:\, \# I(x)\ge 3\},
\]
where
\begin{equation}\label{closechambers}
I(x)=\left\{i\in\{0,\ldots,m\}\,:\, 0<\liminf_{r\to 0}|B_r(x)\cap\mathcal E(i)|r^{-N}\right\}.
\end{equation}

For the sake of completeness we state here below a result that allows us to reformulate the definition of the $c$-singular set and we refer to \cite{CLM} and \cite[Chapter 30]{Ma} for a proof.\\

\begin{lem}[Infiltration Lemma]\label{infiltrationlem}
 Let $\mathcal E$ be a $(\Lambda,\rho)$-minimizing cluster, then there exist constants $\varepsilon$ and $r_0$ depending only on $\Lambda$, $\rho$ and $N$ such that if
\[
|\mathcal E(i)\cap B_r(x)|<\varepsilon r^N\omega_N,
\]
for some $x\in\R^N$, $i=0,\ldots,m$ and $r<r_0$, then
\[
|\mathcal E(i)\cap B_{r/2}(x)|=0.
\]
\end{lem}

\begin{pro}\label{csingpro}
 Let $\mathcal E$ be a $(\Lambda,\rho)$-minimizing cluster, let $\varepsilon$ and $r_0$ be the constants defined after Lemma \ref{infiltrationlem}. Then $x\in c\mathrm{-sing}(\mathcal E)$ if and only if $x\in\partial E$ and
\begin{equation}\label{csing}
|\mathcal E(i)\cap B_r(x)|+|\mathcal E(j)\cap B_r(x)|\le(1-\varepsilon)r^N\omega_N,
\end{equation}\label{csingeq}
for every $i\neq j$ and every $r\le r_0$.
\end{pro}
\begin{proof}
Suppose that $x$ is such that 
\[
|\mathcal E(i)\cap B_r(x)|+|\mathcal E(j)\cap B_r(x)|>(1-\varepsilon)r^N,
\]
for some $r<r_0$ and $i\neq j$; then, for every $h\neq i\neq j$, since the chambers are disjoint, one has $|\mathcal E(h)\cap B_r(x)|<\varepsilon r^N\omega_N$. By Lemma \ref{infiltrationlem} $|\mathcal E(h)\cap B_{r/2}(x)|=0$, and then $h\notin I(x)$. In other words $I(x)\subset\{i,j\}$, thus $x\notin c-{\mathrm sing}(\mathcal E)$.\\

We are left to show that if $x$ satisfies \eqref{csing}, then $x\in c-{\mathrm sing}(\mathcal E)$. Suppose that $I(x)=\{i,j\}$ and let $r_n$ be a sequence such that $r_n\to 0$ and $\lim_{n\to\infty}|B_{r_n}(x)\cap\mathcal E(h)|r_n^{-N}=0$, for every $h\neq i\neq j$.
Then
\[
\omega_N=\lim_{n\to\infty}\left|B_{r_n}(x)\cap \left(\bigcup\mathcal E(l)\right) \right|r_n^{-N}=\lim_{n\to\infty}\sum |B_{r_n}(x)\cap\mathcal E(l)|r_n^{-N}
\]
\[
=\lim_{n\to\infty}\left[|B_{r_n}(x)\cap\mathcal E(i)|r_n^{-N}+|B_{r_n}(x)\cap\mathcal E(j)|r_n^{-N}\right].
\]
Thus \eqref{csing} cannot be satisfied by the chambers $\mathcal E(i)$ and $E(j)$, provided $r$ is sufficiently small.
\end{proof}

We now introduce the {\it perimeter density} $\Theta_x(\mathcal E)$ as follows:
\[
\Theta_x(\mathcal E)=\lim_{r\to 0}\frac{P(\mathcal E;B_r(x))}{\omega_{N-1} r^{N-1}}
\]
Existence and finiteness of $\Theta_x(\mathcal E)$ for $(\Lambda, \rho)$-minimizing clusters is given by the so-called {\it monotonicity formula} stating that the quantity
\begin{equation}\label{monoton}
e^{\Lambda r}\frac{P(\mathcal E;B_r(x))}{\omega_{N-1} r^{N-1}}
\end{equation}
is increasing, for every $r<\rho$ and $x\in\partial \mathcal E$, see for instance \cite[Theorem 17.6]{SimonBook}\\
In particular we can infer from the monotonicity formula \eqref{monoton} the existence of blow-ups at every point $x\in\partial\mathcal E$, namely the sets $\mathcal E_{x,r}=(\mathcal E-x)/r$ converge, as $r\to 0$, in $L^1$ to a cluster $\mathcal E_{x,0}$ and $\pa\mathcal E_{x,r}$ converge in Hausdorff to the boundary of $\mathcal E_{x,0}$.  By investigation of the precise error term in the monotonicity formula, one obtains that the limit set is a conical cluster which minimizes the perimeter in a sense that will be made rigorous by the following definition. 
\par
We say that $\mathcal E$ is a {\it cone-like} $m$-cluster if, for $i=0,\ldots,m$, $\mathcal E(i)$ is an open cone with vertex at the origin of $\R^N$. A cone-like $m$-cluster $\mathcal E$ is a {\it minimizing cone-like cluster} provided 
\[
P(\mathcal E; B_r(x))\le P(\mathcal F; B_r(x)),
\]
for every $r>0$, every $x\in\R^N$ and every $m$-cluster $\mathcal F$ such that $\mathcal E(i)\Delta\mathcal F(i)\subset\subset B_r(x)$.

\section{Connectedness of the interfaces}\label{sec.connected}
The aim of this section is to proof the following lemma:
\begin{lem}\label{lem.connected}
Let $N>2$ and $\mathcal E$ be a minimizing cone-like cluster, then the set $\partial\mathcal E\cap\mathbb S^{N-1}$ has the following property: if $A$ and $B$ are non-empty sets such that $\partial\mathcal E\cap \mathbb S^{N-1}=A\cup B$, then $\mathrm{dist(A,B)=0}$.
\end{lem}
We will derive it as a consequence of the following two statements. 
The first one is a local version of a generalized version of Frankel's Theorem, \cite{FR}. The second one is a regularity result that enables us to apply the first one. 
\begin{lem}[local version of {\cite[Theorem 3]{PW}}]\label{lem.local version}
Let $M^n$ be a complete, connected Riemannian manifold of positive Ricci curvature. If $\Sigma_1, \Sigma_2$ are two minimal hypersurfaces inside $M^n$ possibly with boundary then the map
\[ (p,q) \in \Sigma_1 \times \Sigma_2 \mapsto \dist_{M^n}(p,q) \]
cannot have a local interior minimum. 
\end{lem}

\begin{proof}
The lemma is proven in precisely the same way as the "global" version presented in \cite[Proof of Theorem 3]{PW} because it is a local argument. It is a direct consequence of Synge's second variation formula for geodesics.
\end{proof}

\begin{lem}\label{lem.regular touching points}
Let $\mathcal{E}$ be a $(\Lambda, \rho)$-minimizing cluster, and $x \in \partial \mathcal{E}$ with the property that there exists a ball $B_r(p) \subset \R^N$ with $x \in \partial B_r(p)$ and $\partial \mathcal{E} \cap B_r(p) = \emptyset$ then $x$ is a regular point.  
\end{lem}

Before giving their proofs let us conclude the proof of Lemma \ref{lem.connected}.
\begin{proof}[Proof of Lemma \ref{lem.connected}]
Assume by contradiction that $\dist(A,B)>0$. Since $\partial \mathcal{E}$ is closed, $A$ and $B$ are closed, hence there exist $a\in A, b \in B$ such that $\dist_{\Sp^{N-1}}(a,b)=\dist_{\Sp^{N-1}}(A,B)=d$. Let $\gamma: [0,d] \to \Sp^{N-1}$ be the length minimizing geodesic\footnote{The geodesic is unique since $\partial \mathcal{E}$ cannot consists only by two antipodal points for $N>2$, since $\mathcal{H}^{N-2}(\mathcal{E}\cap \Sp^{N-1})>0$.}  between $a,b$. Let $m:=\gamma(\frac{d}{2})$ be the midpoint on $\gamma$. The geodesic ball $\mathcal{B}:=\mathcal{B}_{\frac{d}{2}}(m)$ satisfies $\mathcal{B} \cap A = \emptyset = \mathcal{B}\cap B$ otherwise it would contradict $\dist_{\Sp^{N-1}}(A,B)=d$. 
Since $\mathcal{E}$ is cone like we conclude that $\mathcal{E} \cap C_\mathcal{B} = \emptyset$ where $C_\mathcal{B}$ is the cone over $\mathcal{B}$ i.e. $C_\mathcal{B}:=\{ \lambda y \colon y \in \mathcal{B}, \lambda \in \R^+ \}$. By construction $a, b \in \partial C_\mathcal{B}$ and $\partial C_{\mathcal{B}}$ is a smooth hypersurface outside of $0$. We conclude that there are euclidean balls $B_1, B_2 \subset C_{\mathcal{B}}$ with $a \in \partial B_1, b \in \partial B_2$. Hence we are in the situation of Lemma \ref{lem.regular touching points} at $a,  b$. This implies that there exists $\varepsilon>0$ s.t. $\Sigma_1:=A \cap B_{\varepsilon}(a), \Sigma_2:=B\cap B_{\varepsilon}(b)$ are $C^{1,\gamma}$-regular. Furthermore since $\mathcal{E}$ is a minimizing cone-like cluster we deduce that $\Sigma_i$ are minimal hypersurfaces with boundary in $\Sp^{N-1}$ for $i=1,2$. By construction the map 
\[ (p,q) \in \Sigma_1 \times \Sigma_2 \mapsto \dist_{\Sp^{N-1}}(p,q) \]
takes a local minimum in $a,b$. This contradicts Lemma \ref{lem.local version} since $\operatorname{Ric}_{\Sp^{N-1}} > 0 $ 
 for $N>2$ and proves the lemma.
\end{proof}

\begin{proof}[Proof of Lemma \ref{lem.regular touching points}]
As shown in   about the existence of tangent cone-like minimizing cluster one has the following equivalence:
\begin{itemize}
\item[(i)] $x \in \partial \mathcal{E}$ is a regular point, i.e. there exists a neighborhood of $x$ where $\partial \mathcal{E}$ is an embedded $C^{1,\gamma}$-hypersurface;
\item[(ii)] $x$ has a "halfspace" as a tangent cone, i.e. there exists a sequence $r_k \to 0$ and $i\neq j$ such that up to rotation $\frac{\mathcal{E}-x}{r_k} \to \mathcal{K}$ with $\mathcal{K}(i)=\R^N_+, \mathcal{K}(j)=\R^N_-$ and $\mathcal{K}(h)=\emptyset$ for all $ h \neq i,j$.
\end{itemize}
Once again by \cite[Theorem 4.13]{CLM} for every sequence $r_k \to 0$ there is a subsequence, still denoted by $r_k$, a cone-like minimizing cluster $\mathcal{K}$ such that 
\[ \frac{\mathcal{E}-x}{r_k} \to \mathcal{K}. \]
Up to a rotation we have $\frac{B_r(p)-x}{r_k} \to \R^N_-$ as $k \to \infty$. Hence $\partial\mathcal{K} \subset \overline{\R^N_+}$.
The Lemma follows showing the following claim:\\

\emph{Claim: }every cone-like minimizing cluster $\mathcal{K}$ with $\partial \mathcal{K}$ contained in a halfspace is a "halfspace".\\

\emph{Proof of the claim:} After reordering the components of $\mathcal{K}$ we may assume that $\R^N_- \subset \mathcal{K}(1) $ and $\mathcal{K}(i) \subset \overline{\R^N_+}$ for all $i>1$. Fix $\varphi, f \in C^1(\R, \R)$ non-negative, non increasing with $\varphi(s), f(s) = 1$ for $s\le 0$ and define the vectorfield 
\[ X(x):= \varphi(\abs{y}) f(t) e_{N}. \]
Observe that $\spt(X) \cap \overline{R^N_+} \subset B_R $ for some $R>0$. Hence if $\Phi_t$ denotes the flow generated by $\Phi$ we have that 
\[ \mathcal{K}(j) \Delta \Phi_t(\mathcal{K}(j)) \subset B_R \text{ for all } j = 1, \dotsc, m, \abs{t}<\delta. \] 

Since $\mathcal{K}$ is minimizing and the first variation of perimeter, \cite[Theorem 17.5]{Ma}, and \ref{eq:1/2perimeter} we must have 
\[ 0 = \frac{d}{dt}|_{t=0}  P(\Phi_t(\mathcal{K})) = \frac{1}{2} \sum_{i=1}^m \int_{\partial \mathcal{K}(i)} \Div_{\mathcal{K}(i)}(X) \, d\mathcal{H}^{N-1}. \]
Since $DX(x) = \varphi(\abs{y}) f'(t) e_N \otimes e_N + \varphi'(\abs{y}) \frac{f(t)}{\abs{y}} e_N \otimes y$, the boundary divergence of $X$ on $\mathcal{K}(i)$ is given by 
\begin{align*}
\Div_{\mathcal{K}(i)}(X)= \varphi(\abs{y}) f'(t) \left( 1 - \langle \nu_i, e_N \rangle^2 \right) + \varphi'(\abs{y}) \frac{f(t)}{\abs{y}} \left( \langle e_N, y\rangle - \langle \nu_i, e_N \rangle \langle y, \nu_i \rangle \right)\\
= \varphi(\abs{y}) f'(t) \left( 1 - \langle \nu_i, e_N \rangle^2 \right) + \varphi'(\abs{y}) \frac{f(t)}{\abs{y}} t \langle \nu_i, e_N\rangle^2 
\end{align*}
where we used beside $\langle e_N, y \rangle =0 $ that, since $\mathcal{K}(i)$ is an open cone, 
\[0=\langle x, \nu_i \rangle = \langle y , \nu_i\rangle + t \langle \nu_i, e_N \rangle .\]
This implies that $\Div_{\mathcal{K}(i)}(X) \le 0$ for each $i$. Choosing $\varphi$ and $f$ appropriate we deduce that $\partial \mathcal{K}(i) \subset \{ t= 0\} $ for all $i$. This concludes the proof of the claim. 
\end{proof}

\section{Proof of the lower density bound}\label{sec.ldb}
In this section we prove Theorem \ref{thm.lowdens} for minimizing cone-like $m$-clusters. Before proving our main result we show in the following lemma that, if $m\ge 2$, then the origin cannot be the only $c$-singular point.
\begin{lem}\label{lem.csingsphere}
Let $N>2$, $m\ge 2$ and let $\mathcal E$ be a minimizing cone-like cluster, then $c-\mathrm{sing}(\mathcal E)\cap \mathbb S^{N-1}\neq\emptyset$. 
\end{lem}
\begin{proof}
Fix a chamber $\mathcal E(i)$; for $x\in\partial\mathcal E(i)$ let $I(x)$ be the set defined in \eqref{closechambers}. Suppose that $c-\mathrm{sing}(\mathcal E)\cap \mathbb S^{N-1}\cap\partial \mathcal E(i)=\emptyset$, then $I(x)=\{i,j(x)\}$, for some $i\neq j(x)\in\{0,\ldots,m\}$.
Let 
\[
r(x)=\sup\left\{r\ge0\,:\,|B_r(x)\cap\mathcal E(l)|=0, \ \ \text{for every } l\notin I(x)\right\}
\]
Thanks to Lemma \ref{infiltrationlem} we have that $r(x)>0$, for every $x\in\partial\mathcal E(i)$.
Moreover $r$ is a Lipschitz continuous function, indeed, since $B_{r-\varepsilon}(y)\subset B_r(x)\subset B_{r+\varepsilon}(y)$, whenever $|x-y|<\varepsilon$, thus $r(x)-\varepsilon< r(y)<r(x)+\varepsilon$. 
With the same argument it is easy to check that $j(x)$ is locally constant.
\par
Let $K$ be a connected component of $\mathcal E(i)\cap\mathbb S^{N-1}$ and let $j=j(x)$, for every $x\in K$; since $K$ is a compact set, then the function $r(x)$ achieves a positive minumum $\overline r$. We get the contradiction by using Lemma \ref{lem.connected} with $A=K$ and $B=\partial\mathcal E\cap\mathbb S^{N-1}\setminus K$. Indeed $\mathrm{dist}(A,B)>\overline r$ and $B\neq \emptyset$, since $\partial \mathcal E(l)\subset B$, for every $l\notin\{i,j\}$.
\end{proof}

\begin{proof}[Proof of Theorem \ref{thm.lowdens}]
We are going to prove the statement by induction on the ambient dimension $N$.\\

If $N=2$ then the estimate follows by the fact that the boundary any minimizing cone-like cluster consist of three line segments meeting at the origin.\\

Let $N>2$, let $\mathcal E$ a cone-like minimizing cluster of $\R^N$ and suppose that our statement is valid for every cone-like minimizing cluster $\mathcal E'\subset \R^{N-1}$.\\
Thanks to Lemma \ref{lem.csingsphere} there exists a point $x\in c-\mathrm{sing}(\mathcal E)\cap\mathbb S^{N-1}$ that we can assume, up to rotating $\mathcal E$, that coincides with $e_N$.
Since $\mathcal E$ is conical, then, for $\varepsilon>0$, every point $\varepsilon e_N$ belongs to $c-\mathrm{sing}(\mathcal E)$, thus monotonicity formula \eqref{monoton} and the fact that $\mathcal E$ is left invariant under blow ups at the origin imply that
\[
\Theta_0(\mathcal E)\ge \Theta_{e_N}(\mathcal E).
\]
Indeed 
\begin{align*}
\Theta_0(\mathcal E) &=\frac{P(\mathcal E;B_1(0))}{\omega_{N-1}}=\frac{P(\mathcal E;B_1(0))}{\omega_{N-1}(1-\varepsilon)^{N-1}}+O(\varepsilon)\\
&\ge \frac{P(\mathcal E; B_{1-\varepsilon}(\varepsilon e_N))}{\omega_{N-1}(1-\varepsilon)^{N-1}}+O(\varepsilon)\ge \Theta_{e_N}(\mathcal E)+O(\varepsilon).
\end{align*}

We set $\tilde{\mathcal E}=\lim_{r\to 0}\mathcal E_{e_N,r}$; as already mentioned $\tilde{\mathcal E}$ is a cone-like minimizing cluster with $\Theta_0(\tilde{\mathcal E})=\Theta_{e_N}(\mathcal E)$, moreover it is easy to check that $0\in c-\mathrm{sing}(\tilde{\mathcal E})$.\\
Since $\mathcal E$ is conical, then $\tilde{\mathcal E}$ is invariant under translation in the $e_N$-direction, namely $\tilde{\mathcal E}=\mathcal E'\times\R$, for some (still minimizing) $m'$-cluster $\mathcal E'\subset \R^{N-1}$. Again $0\in c-\mathrm{sing}(\mathcal E')$ and thence with $m'\ge 2$.\\
The inductive assumption gives that $\Theta_0(\mathcal E')\ge 3/2$.
We can now conclude by computing the density $\Theta_0(\tilde{\mathcal E})$ in terms of $\Theta_0(\mathcal E')$. Indeed the {\it co-area} formula for $\mathcal H^{N-1}$-rectifiable sets yields
\begin{equation}\label{eq:fubini}
\begin{split}
\Theta_0(\tilde{\mathcal E})
&=\frac{P(\tilde{\mathcal E};B_1)}{\omega_{N-1}}
\ge\frac{1}{\omega_{N-1}}\int_{-1}^1 \mathcal H^{N-2} \left(\partial^*\tilde{\mathcal E}\cap B_1\cap\{x_N=t\}\right)dt\\
&=\frac{1}{\omega_{N-1}}\int_{-1}^1 \mathcal H^{N-2}\left(\partial^*\mathcal E'\cap B_{\sqrt{1-t^2}}\right)dt\\
&=\frac{1}{\omega_{N-1}}\int_{-1}^1\mathcal H^{N-2}\left(\partial^*\mathcal E'\cap B_1\right)(1-t^2)^{\frac{N-2}{2}}dt\\
&=\frac{P(\mathcal E';B_1)}{\omega_{N-1}}\int_{-1}^{1}(1-t^2)^{\frac{N-2}{2}}dt=\Theta_0(\mathcal E')\frac{\omega_{N-2}}{\omega_{N-1}}\int_{-1}^{1}(1-t^2)^{\frac{N-2}{2}}dt\\
&\ge\frac 3 2 \frac{\omega_{N-2}}{\omega_{N-1}}\int_{-1}^{1}(1-t^2)^{\frac{N-2}{2}}dt=\frac 3 2,
\end{split}
\end{equation}
where the last equation easily follows from Fubini's Theorem.

\end{proof}

\begin{remark}\label{rmk.subsequences}
{\rm In the proof above, we showed that there exist $x_k\in c\mathrm{-sing}(\mathcal E)$ and $r_k>0$ such that $\mathcal E_k=\mathcal E_{x_k,r_k}\rightarrow \mathcal E'\times \R$, such that $\mathcal E'$ is an $N-1$-dimensional minimising cone-like cluster, singular at the origin. In particular, arguing by induction on the dimension $N$, and repeating the construction of the proof of Theorem \ref{thm.lowdens}, it is possible to show that there exists a sequence of points $x_k\rightarrow 0$, $x_k\in c\mathrm{-sing}$,  and  a sequence of positive numbers $r_k\rightarrow 0$, such that, up to rotations
\[
\frac{\mathcal E-x_k}{r_k}\rightarrow Y\times \R^{N-2}.
\]
}
\end{remark}

\section{Consequences}\label{sec.consequences}
\subsection{Characterization of $Y \times \R^{N-2}$ by its density }
\begin{cor}\label{cor.characterizationbydensity}
Suppose $\mathcal{K}$ is a minimzing conelike cluster with $\Theta_0(\mathcal{K})=\frac32$ then up to rotation we have
\[ \mathcal{K}=Y \times \R^{N-2}.\]
\end{cor}
\begin{proof}
To show the claim we combine Proposition \ref{csingpro} with some classical consequences of the monotonicity of the density:\\
Let $\mathcal{E}$  be a any minimizing cluster then 
\begin{enumerate}
\item if $\frac{ P(\mathcal{E};B_r(y))}{w_{N-1}r^{N-1}}= \frac{ P(\mathcal{E};B_s(y))}{w_{N-1}s^{N-1}}$ for some $0\le r < s$ then $\mathcal{E}$ coincides with a cone in the annulus $B_{s}(y)\setminus B_r(y)$;
\item if $\mathcal{E}$ is a cone with vertex $0$ then $\Theta_y(\mathcal{E}) \le \Theta_0(\mathcal{E})$ for all $y$;
\item additionally one has that $L_{\mathcal{E}}:=\{ y \colon \Theta_y(\mathcal{E}) = \Theta_0(\mathcal{E})\}$ is a linear subspace of $\R^N$. Suppose $k= \operatorname{dim}(L_{\mathcal{E}})$ then there is a cone-like minimizing cluster $\mathcal{E}'$ in $\R^{N-k}$ such that $ \mathcal{E} = \mathcal{E}' \times L_{\mathcal{E}}$.
\end{enumerate}
(1) follows from the montonicity formula \cite[Theorem 28.9]{Ma}, (2) and (3) can be found for instance in a more general setting in \cite{Wh}.\\
We will show the corollary by induction on the dimension $N$. For $N=2$ the statement follows by the classification of cone-like clusters in $\R^2$, \cite[Proposition 30.9]{Ma}. Suppose the corollary is proven for dimension $N'<N$. Let $\mathcal{K}$ be a minimizing cone-like cluster in $\R^N$ with $\Theta_0(\mathcal{K})=\frac32$. By Lemma \ref{lem.csingsphere} there is $y \in c-\sing(\mathcal{K}) \cap \Sp^{N-1}$. Applying Theorem \ref{thm.lowdens} we have $\Theta_y(\mathcal{K}) \ge \frac{3}{2}$. Combining (2) and (3) above we deduce that $k:=\operatorname{dim}(L_{\mathcal{K}})\ge 1$ and the existence of cone-like minimizing cluster $\mathcal{K}'$ in $\R^{N-k}$ with  $ \mathcal{K} = \mathcal{K}' \times L_{\mathcal{K}}$. By Fubini's Theorem, compare \eqref{eq:fubini}, we have 
\[ \frac3 2\le\Theta_0(\mathcal{K}')\le\Theta_0(\mathcal{K})\le \frac3 2. \]
Hence by $\mathcal{K}'$ satisfies the conditions of the corollary and by induction hypothesis we have $\mathcal{K}' = Y \times \R^{N-k-2}$ and so $\mathcal{K}= Y \times \R^N$.
\end{proof}
\subsection{Sequential convergence of clusters}
The aim of this subsection is to proof the following:
\begin{pro}\label{prop.sequential}
Let $\mathcal{E}_k$ be a sequence of $(\Lambda_k, \rho_k)$-minimizing clusters with $\Lambda_k \to \Lambda, \rho_k \to \rho_0>0$ and $\mathcal{E}_k \to \mathcal{E}$ in some open set $U \in \R^{N-1}$ as $k \to \infty$, then the following holds:
\[ c-\sing(\mathcal{E}_k ) \to c-\sing(\mathcal{E}) \text{ in Hausdorff on every } V \Subset U. \]
\end{pro}
\begin{proof}
We will first show the easier part that for each $\varepsilon >0$ there exists $k_0>0$ such that 
\begin{equation*}\label{eq:easierdirection}
 c-\sing(\E_k) \subset (c-\sing(\E))_\varepsilon.
\end{equation*} 
Suppose the inclusion fails, then there exist a sequence $x_k \in c-\sing(\E_k)\cap U'$ with $\dist(x_k, c-\sing(\E))> \varepsilon$. Passing to an appropriate subsequence we have $x_k \to x \in \overline{U'\setminus (c-\sing{\E})_\varepsilon} $ i.e. $x \notin c-\sing(\E)$. Hence by Proposition \ref{csingpro} there exists a tuple $i\neq j$,  a radius $r>0$ and a positive number $\varepsilon$ such that 
\[ \abs{\E(i) \cap B_r(x)} + \abs{\E(j) \cap B_r(x)} > (1-\varepsilon) r^N w_N .\]
But since $\E_k \to \E$, and $B_r(x_k) \to B_r(x)$ for $k$ sufficient large we have 
\[ \abs{\E_k(i) \cap B_r(x_k)} + \abs{\E_k(j) \cap B_r(x_k)} > (1-2\varepsilon) r^N w_N. \]
But again by Proposition \ref{csingpro} this implies $x_k \notin c-\sing(\E_k)$. Which is the desired contradiction.\\

We are left to show the harder part of the proposition, namely that if $x \in c-\sing(\mathcal{E})\cap U$ then there exists a sequence $x_k \in c-\sing(\mathcal{E}_k)$ with $x_k \to x$. This will be achieved by the fact that a $Y\subset \R^2$ is essentially non removable:

\begin{lem}\label{lem.Ynonremovable}
Given a partition of $B_1\subset \R^2$ into three sets of finite perimeter $E_1,E_2,E_3$ such that $E_i \cap \partial B_1$ is a single interval for each $i=1,2,3$ then
\[ \partial E_1 \cap \partial E_2 \cap \partial E_3  \neq \emptyset. \]
\end{lem}
Before we give the proof of this statement let us show how to conclude. \\
We will use the following notation for cylindrical sets: 
\[
C_r:=B_r\times ]-r,r[^{N-2} \subset \R^2 \times \R^{N-2} \ \   \text{and}   \ \ S_\varepsilon:= B_\varepsilon \times \R^{N-2}.
\]
By the Remark \ref{rmk.subsequences}, there exists sequences $x_l \to x$ and $r_l \to 0$ such that up to rotation we have 
\[ \frac{\mathcal{E}- x_l}{r_l} \to Y \times \R^{N-2} \text{ as } l \to \infty. \]
Furthermore up to relabeling the chambers and an application of the infiltration lemma \ref{infiltrationlem} we may assume that 
\[ \frac{\mathcal{E}(i)-x_l}{r_l}\cap C_5 = \emptyset \text{ for all } i > 3, l >l_0. \]
Since $\mathcal{E}_k \to \mathcal{E}$ as $k \to \infty$ we can therefore find a  sequence $\{k(l)\}_{l \in \N}$ such that 
\[ \mathcal{E}_l':=  \frac{\mathcal{E}_{k(l)} - x_l}{r_l} \to Y\times \R^{N-2} \text{ as } l \to \infty. \]
Hence we will have that $\mathcal{E}'_l(i) \cap C_5 = \emptyset $ for all $i>3$ and $l>l_1$. So will forget about all higher $i>3$ in sequel and consider $\mathcal{E}'_l$ as a cluster with $3$ chambers.\\
Observe that $\sing(Y\times \R^{N-2}) = \{0\}\times \R^{N-2}$. The small excess regularity criterion, e.g. \cite[Theorem 4.1]{CLM}, implies that for $l>l_2$ we have 
\begin{align}\label{eq:closetoY}
\sing(\mathcal{E}'_l)\cap C_4 & \subset S_\epsilon  \nonumber \\ \partial \mathcal{E}'_l\cap C_3 \setminus S_\epsilon & \text{ is a $C^{1,\alpha}$- graph over parts of $Y\times \R^{N-2}$ }.
\end{align}
To conclude the claim it is sufficient to show what for each $l>l_2$ one has $c-\sing(\mathcal{E}_l')\cap C_2 \neq \emptyset$. To do so fix some $l>l_2$. Since the argument is independent of $l$ we drop the index and write only $\mathcal{E}'$.\\
Assume by contradiction that $c-\sing(\E')\cap C_2 = \emptyset$. 
We consider the the sliced cluster for $z \in ]-2,2[^{N-2}$
\[ \mathcal{E}'_z:= \mathcal{E}'\cap C_2 \cap \R^2 \times \{z\} \subset B_2 \subset \R^2. \]
Observe that for almost every $z$ we have 
\begin{enumerate}
\item $ \mathcal{E}'_z= \left\{ \E_z'(1), \E_z'(2), \E_z'(3) \right\} $ is a partition of $B_2\subset \R^2$ by sets of finite perimeter.
\item as a consequence of \eqref{eq:closetoY} $\partial \E'_z \setminus B_{2\epsilon}$ consists of three $C^{1,\alpha}$-curves and $\E'_z(i)\cap \partial B_1$ is a simple interval for $i=1,2,3$. 
\end{enumerate}
Fix any such $z\in ]-2,2[^{N-2}$. We assumed that $c-\sing(\E)\cap C_2 = \emptyset$. But then the Proposition \ref{csingpro} implies that for each $(y,z)\in C_2$ there exists a radius $r = r(y,z) >0$ and and index $i =  i(y,z) \in \{ 1, 2,3 \}$ such that
\[ \E'(i) \cap B_{r}((y,z)) = \emptyset,\]
which implies 
\[ \E'_z(i) \cap B_r(y) = \emptyset. \]
Observe that this equivalent to $\partial \E_z'(y) \cap B_r(y) = \emptyset$.  As a consequence we have
\[ \partial \E'_z(1) \cap \partial \E'_z(2) \cap \partial \E'_z(3) = \emptyset. \]
But this contradicts Lemma \ref{lem.Ynonremovable}, because the assumptions are satisfied by (2) above.
\end{proof}

To prove Lemma \ref{lem.Ynonremovable} we will use the following "path-connectedness" property of sets of finite perimeter in the plane, Figure 1.
\begin{figure}\label{fig.fig1}
\caption{}
    \includegraphics[width=0.2\textwidth]{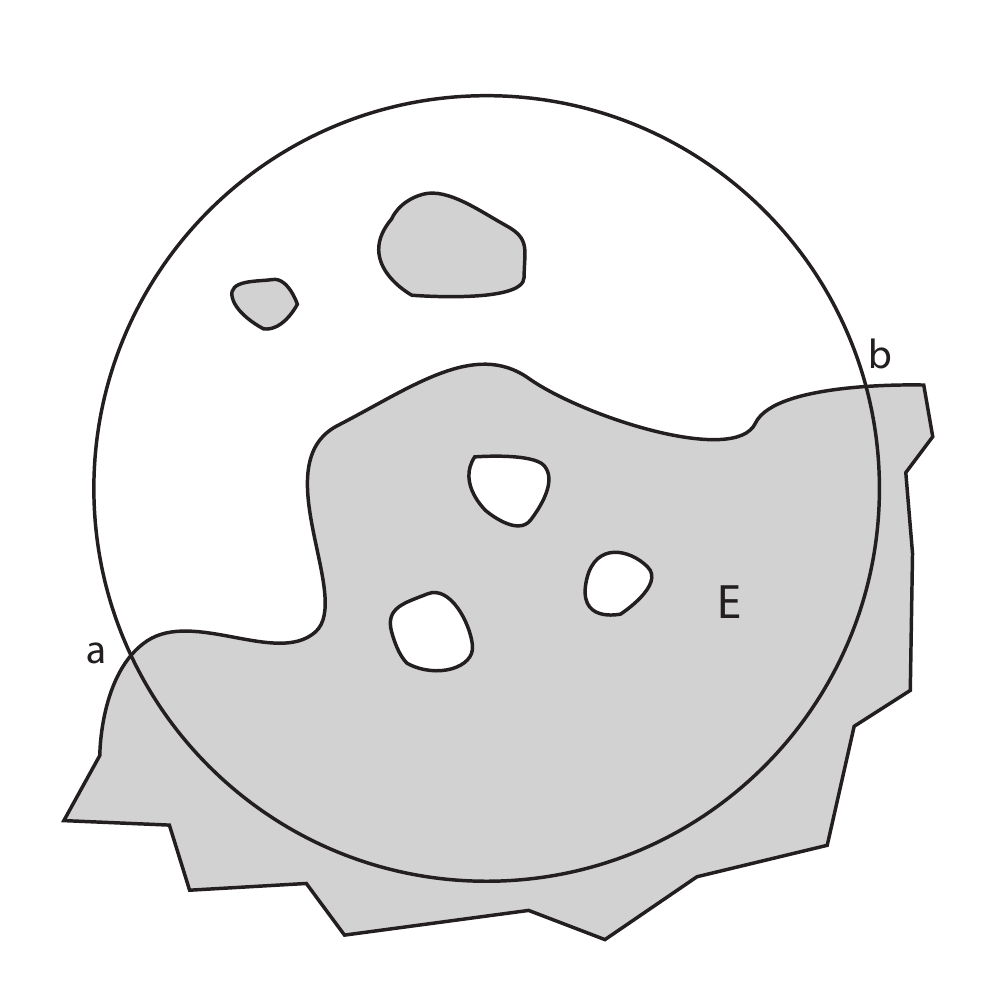}
\end{figure}
\begin{lem}\label{lem.pathconnected}
Let $E\subset \R^2$ be a set of finite perimeter s.t. $E \cap \partial B_1$ is a single interval with $\partial E \cap \partial B_1 =\{a, b\}$, then there exists a rectifiable curve $\gamma \subset B_1$ s.t. $\partial \gamma = \{a, b\}$ and $\gamma \subset \partial E$ i.e. $\{a,b\}$ are path connected inside $\partial E$.
\end{lem}
\begin{proof} The proof is based on an approximation by smooths sets and the classification of compact 1-manifolds. \\
Passing to 
\[ \mathbf{1}_{\hat{E}}(x) = \begin{cases} \mathbf{1}_E(x) & \text{ for } \abs{x} \le 1 \\ \mathbf{1}_E(\frac{x}{\abs{x}}) & \text{ for } \abs{x} > 1 \end{cases} \]
we may assume that $E$ is a cone outside of $B_1$ i.e. $\partial E \cap \partial B_R = \{ Ra , Rb \}$ for all $R\ge 1$. \\
Given $\varepsilon>0, \epsilon_n \to 0$ and $t \in ]0,1[$ we set for an symmetric mollifier $\rho_\varepsilon$
\[ u_\varepsilon= \rho_\varepsilon \star \mathbf{1}_E \quad u_n= u_{\varepsilon_n} \quad E_n^t=\{ u_n > t \}. \]
As shown in \cite[Theorem 13.8]{Ma} we have that for a.e. $t \in ]0,1[$ we have
\begin{enumerate}
\item $t$ is a regular value for $u_n$;
\item $E_n^t \xrightarrow{loc} E$, $P(E_n^t, B_2) \to P(E, B_2)$;
\item $\partial E_n^t \subset (\partial E)_{\varepsilon_n}$.
\end{enumerate}
Appealing once more to the Morse-Sard theorem we may assume that 

\begin{enumerate}\setcounter{enumi}{3}
\item $t$ is a regular value for $u_n|_{\partial B_2}: \partial B_2 \to [0,1]$.
\end{enumerate}

We claim that for $n> n_1$ we have
\[\partial E_n^t \cap \partial B_2 =\{ 2a_n, 2b_n\} \quad a_n \to a, b_n \to b.\]
This can be seen as follows. Let $S$ be the set $\{ \lambda x \colon x \in E\cap \partial B_1, \lambda >0 \}$ then for $n$ large we have 
\[ u_n(x) = \rho_{\epsilon_n} \star \mathbf{1}_S(x), x \in \partial B_2. \]
The claim easily holds true for $\rho_{\varepsilon_n} \star \mathbf{1}_S(x)$. \\
Fix $0< t <1$ such that (1) to (4) above are satisfied for all $n \in \N$. By the classification of compact 1-dimensional manifolds, \cite[Appendix]{Mi}, $\partial E_n^t \cap B_2 = \{u_n =t\}\cap B_2$ is a finite disjoint union of circles and segments. Since $\partial E_n^t \cap \partial B_2 = \{2a_n, 2b_n\}$ there is precisely one segment $\gamma_n$. This segment satisfies 
\[ \partial \gamma_n =\{ a_n, b_n \} \quad \operatorname{Length}(\gamma_n) \le P(\partial E_n^t, B_2). \]
Since $P(\partial E_n^t, B_2) \le 2 P(\partial E, B_2)$ by (2) for sufficient large $n$ and compactness of rectifiable curves there is a subsequence satisfying  
\[ \gamma_n \to \gamma \quad \partial \gamma_n \to \{ 2a, 2b\}.\]
It remains to be proven that $\gamma \subset \partial E$. Assume by contradiction that there is $x \in \gamma \setminus \partial E$ then there exist $r>0$ s.t. for $n$ large with $\epsilon_n< \frac{r}{2}$
\begin{align*}
\text{either } \abs{ B_r(x)\cap E} &= 0 \quad u_n(y) = 0 \quad \forall y \in B_{\frac{r}{2}}(x) \\
\text{or } \abs{ B_r(x)\setminus E} &= 0 \quad u_n(y) = 1 \quad \forall y \in B_{\frac{r}{2}}(x).
\end{align*}
In both cases we have $\gamma_n \cap B_{\frac{r}{2}}(x) \subset \{ u_n =t \} \cap  B_{\frac{r}{2}}(x)= \emptyset $, since $0<t<1$. But this contradicts $x \in \gamma$. 
The lemma now follows since $\partial E\setminus \overline{B_1 } =\{ \lambda a \colon \lambda > 1 \} \cup \{ \lambda b \colon \lambda > 1 \}$.
\end{proof}
Now we are able to obtain Lemma \ref{lem.Ynonremovable} as a corollary.
\begin{proof}[Proof of \ref{lem.Ynonremovable}]
Given the sets $E_1,E_2,E_3$; as in Figure 2 we set
\begin{figure}\label{fig.fig2}
\caption{}
    \includegraphics[width=0.2\textwidth]{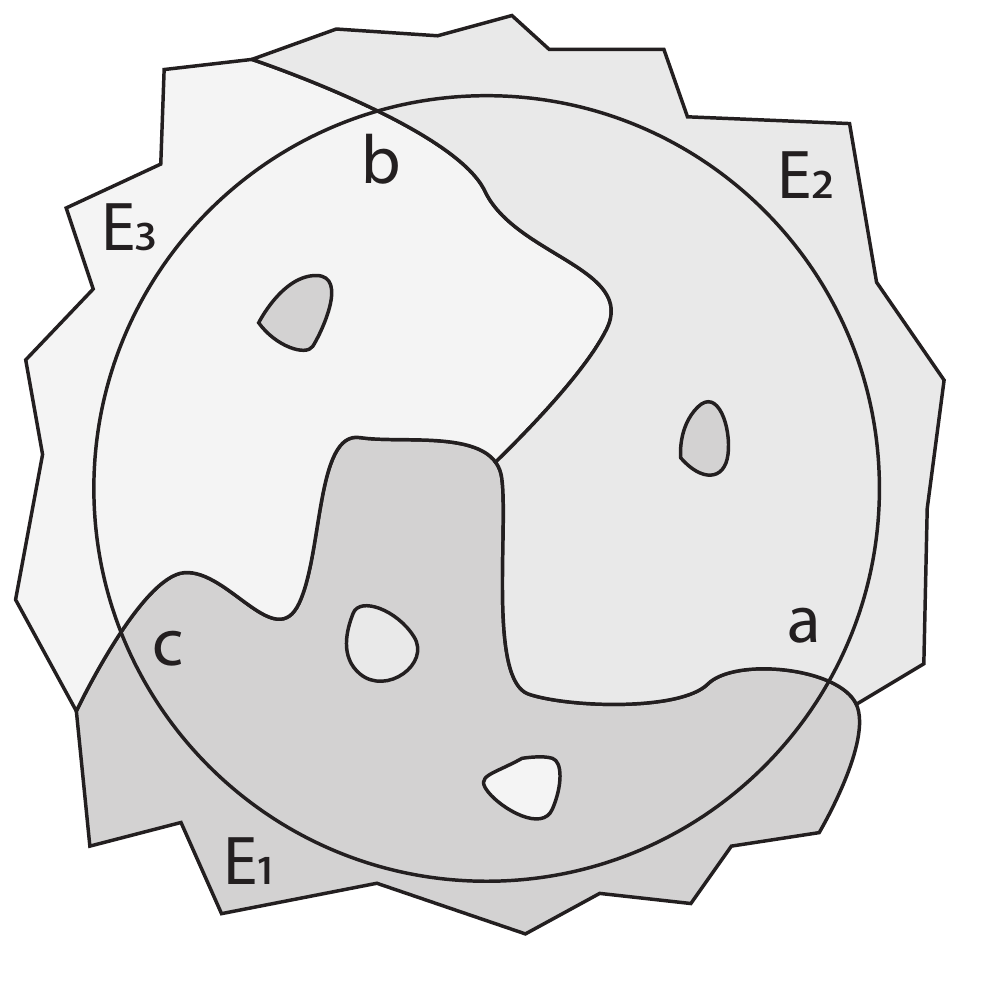}
\end{figure}

\[ a= \partial E_1 \cap \partial E_2\quad  b= \partial E_2 \cap \partial E_3 \quad c= \partial E_1 \cap \partial E_3. \]
We apply lemma \ref{lem.pathconnected} to the set $E_1$ and obtain the existence of a rectifiable curve 
\[ \gamma: [0,1] \to B_1 \text{ with } \gamma \subset \partial E_1, \gamma(0)=a, \gamma(1)=c. \]
By contradiction assume that 
\[ \partial E_1 \cap \partial E_2 \cap \partial E_3  = \emptyset. \]
This implies that
\[ \# I(x) \le 2 \text{ for all } x \in B_1 \]
where 
\[ I(x) = \left\{ i \in \{ 1, 2, 3 \} \colon 0 < \abs{ E_i \cap B_r(x)} < \pi r^2 \; \forall r>0 \right\}. \]
It is straight forward that $I(x)$ is locally constant i.e. the map $t \mapsto I(\gamma(t))$ must be constant. But this is a contradiction since $I(\gamma(0))=\{1,2\}$ and $I(\gamma(1))=\{1,3\}$. Hence the lemma is proven.
\end{proof}

By applying Proposition \ref{prop.sequential} with $\mathcal E_k=\mathcal E$ (and thus $\Lambda_k=\Lambda$ and $\rho_k=\rho$), have the following
\begin{cor}
Let $\mathcal E$ be a $(\Lambda,\rho)$-minimizing cluster, then $c\mathrm{-sing}(\mathcal E)$ is a closed subset of $\partial\mathcal E$.
\end{cor}

\ack
The work of the authors is supported by the MIUR SIR-grant {\it``Geometric
Variational Problems''} (RBSI14RVEZ).

\bibliographystyle{plain}
\bibliography{HirMar.bib}

\end{document}